\documentclass[12pt]{article}
\usepackage{amssymb,amsmath,latexsym,amsthm}

\input xy \xyoption{all}

\newtheorem{thm}{Theorem}[section] \newtheorem{lem}[thm]{Lemma} \newtheorem{prop}[thm]{Proposition} \newtheorem{cor}[thm]{Corollary} 

\theoremstyle{definition}

\theoremstyle{definition}
\newtheorem{remark}[thm]{Remark}

\usepackage[T1]{fontenc}
\usepackage[all,2cell,poly,knot,tips]{xy}

\DeclareMathOperator{\End}{End}

\begin{document}

\title{{\large{\textbf{REMARKS ON UNITS OF SKEW MONOIDAL CATEGORIES}}}}
\author{Jim Andrianopoulos}
\date{\today}
\maketitle
\abstract{This article shows that the units of a skew monoidal category are unique up to a unique isomorphism, and internalises this fact to skew monoidales. Some benefits of certain extra structure on the unit maps are also discussed before the axioms of a skew monoidal category are shown to be independent. }

\section{Introduction}
 
Generalisations of the notion of monoidal category have been studied almost as long as the notion itself; one involves relaxing the invertibility of the maps expressing the associativity and unit conditions. Once invertibility is dropped then what is required is that the directions of these constraints be specified; one such choice leads to the notion of skew monoidal category. What are the particular choices in directions of these constraints, and just as importantly, what particular equations do they satisfy? Mac Lane in \cite{Rice} shows that a list of five axioms is sufficient for the coherence for monoidal categories. Kelly in \cite{Kelly1964} found that there were redundancies in that list and reduced it down to two. However this reduction relied on the invertibility of the associativity and unit maps. In the context of skew monoidal categories no such invertibility is assumed and so we require all of the five axioms of Mac Lane.

One of the first observations about a monoid is that its unit (if it exists) is unique, as shown by the equality $i = i.j = j$. In a monoidal category these equalities become isomorphisms $I \cong I\otimes J \cong J$ ; where now in this context there is also a uniqueness result. We show an analogous result for the units of a skew monoidal category. In this context we no longer have isomorphisms $I \cong I\otimes J$ or $I\otimes J \cong J$ but only the maps $I \longrightarrow I\otimes J$ and $I\otimes J \longrightarrow J$. Thus it might seem that uniqueness up to isomorphism is lost, but surprisingly, it turns out that the composite $I \longrightarrow I\otimes J \longrightarrow J$ {\em is} invertible, and we do still have a uniqueness result for this isomorphism.  

Generalisations of monoidal categories have previously been considered with different choices for the directions of the non-invertible associativity and unit maps, or with fewer coherence data than those considered here.
In \cite{ACU}, Altenkirch, Chapman and Uustalu, while studying relative monads, show a certain functor category is skew monoidal, they call it lax monoidal. Independently, and motivated by bialgebroids, Szlach\'anyi in \cite{Szl2012}, first names and studies skew monoidal categories as such. In this text what we call a skew monoidal category is usually referred to as a left skew monoidal category, and what could have been referred to as a skew psuedomonoid we call a skew monoidale.   

In Section 2 we establish that the units are isomorphic up to a unique isomorphism; this is the analogue for skew monoidal categories of Proposition 1.7 in \cite{K}. This was shown for monoidal categories by Mac Lane in \cite{Rice} and later by Kock in \cite{K}, where references are also given to these results by Saavedra Rivano in \cite{Sav}.
The proofs here follow the same methods employed in \cite{K}. We then impose some extra structure on the unit maps of a skew monoidal category and remark on some consequences of this extra structure. We then depart from the main theme of this article by showing the independence of the axioms for a skew monoidal category before returning to it to conclude the section with some remarks on the unit conditions of a monoidal functor between skew monoidal categories. In Section 3 we internalise the main result of Section 2 to skew monoidales; that is, out of the cartesian monoidal 2-category \textbf{Cat} and lift it into a monoidal bicategory, although by the coherence results of \cite{GPS} it sufficies to work in a Gray monoid.
 
\section{Skew Monoidal Categories}
\subsection{Skew Semimonoidal Categories}

A {\em skew semimonoidal category} is a triple  $({\cal C},\otimes,\alpha)$ where $\cal C $ is a category equipped with a functor $\otimes : \cal C \times \cal C \to \cal C$ (called {\em tensor product}), and a
natural family of {\em lax constraints} $\alpha$ whose components have the directions
\begin{equation*}
\alpha_{X,Y,Z} : (X\otimes Y)\otimes Z \longrightarrow X\otimes (Y\otimes Z)
\end{equation*} subject to the condition that the following diagram commutes
\begin{equation}\label{pentagon}
\xymatrix{
& (W\otimes X)\otimes (Y\otimes Z) \ar[rd]^-{{\alpha_{W,X,Y\otimes Z}}}  & \\
((W\otimes X)\otimes Y)\otimes Z \ar[ru]^-{{\alpha_{W\otimes X,Y,Z}}} \ar[d]_-{{\alpha_{W,X,Y}} \otimes 1_Z} & & W\otimes (X\otimes (Y\otimes Z))  \\
(W\otimes (X\otimes Y))\otimes Z \ar[rr]_-{{\alpha_{W,X\otimes Y,Z}}} & & W\otimes ((X\otimes Y)\otimes Z) \ar[u]_-{1_W\otimes {\alpha_{X, Y,Z}}} }
\end{equation}

\subsection{The Category of Units}\label{Catunits}
A skew monoidal category is a skew semimonoidal category equipped with a chosen unit, in a sense to be defined below. We shall see that if such a unit exists, it is unique up to isomorphism. Furthermore, this isomorphism is compatible, in the sense that it is a morphism in the {\em category of units}, which we now define.

Given a skew semimonoidal category $({\cal C},\otimes,\alpha)$, we form a category $ \cal U(\cal C)$ as follows. The {\em objects} are  triples $(I,\lambda,\rho)$ where $I$ is an object of $\cal C$ and where $\lambda$ and $\rho$ are natural families of {\em lax constraints} whose components have directions 
\begin{equation*}
\lambda_X \colon I\otimes X \longrightarrow X
\end{equation*}
\begin{equation*}
\rho_X \colon X \longrightarrow X\otimes I
\end{equation*}subject to four conditions asserting that the following diagrams commute:

\begin{equation}\label{leftunit}
\xymatrix{
(I\otimes X)\otimes Y \ar[rd]_{\lambda_{X} \otimes 1_Y}\ar[rr]^{\alpha_{I,X,Y}}   && I\otimes (X\otimes Y) \ar[ld]^{\lambda_{X\otimes Y}} \\
& X\otimes Y  &
}
\end{equation}

\begin{equation}\label{midunit}
\xymatrix{
(X\otimes I)\otimes Y \ar[rr]^-{\alpha_{X,I,Y}} && X\otimes (I\otimes Y) \ar[d]^-{1_X\otimes \lambda_Y} \\
X\otimes Y \ar[u]^-{\rho_X \otimes 1_Y} \ar[rr]_-{1_{X\otimes Y}} && X\otimes Y}
\end{equation}

\begin{equation}\label{rightunit}
\xymatrix{
(X\otimes Y)\otimes I \ar[rr]^{\alpha_{X,Y,I}}   && X\otimes (Y\otimes I)  \\
& X\otimes Y \ar[lu]^{\rho_{X \otimes Y}}  \ar[ru]_{1_X\otimes \rho_Y}&}
\end{equation}

\begin{equation}\label{unitunit}
\xymatrix{
I \ar[rd]_{\rho_I}\ar[rr]^{1_I}   && I  \\
& I\otimes I \ar[ru]_{\lambda_I} & .
}
\end{equation}
An {\em arrow} of  $\cal U(\cal C)$ from $(I,\lambda,\rho)$ to $(J,\lambda',\rho')$ is given by an arrow $\varphi \colon I \longrightarrow  J$  in $\cal C$
 such that the following two triangles commute       
\begin{equation}\label{Comp}
\begin{aligned}
\vbox{\xymatrix{
I\otimes X \ar[rd]_{\lambda_{X}}^(0.5){\phantom{aa}}="1"  \ar[rr]^{\varphi\otimes 1_{X}}   && J\otimes X  \ar[ld]^{\lambda'_{X}}_(0.5){\phantom{aa}}="2" \
\\
& X }}
\qquad  \qquad
\vbox{\xymatrix{
X\otimes I \ar[rr]^{1_{X}\otimes \varphi }  && X\otimes J   
\\
& X
\ar[lu]^{\rho_{X}}^(0.5){\phantom{aa}}="1" \
\ar[ru]_{\rho'_{X}}^(0.5){\phantom{aa}}="2" \ }}
\end{aligned}
\end{equation}
The composition of arrows in  $\cal U(\cal C)$ is then given by the composition in $\cal C$.

Given two objects $(I,\lambda,\rho)$ and $(J,\lambda',\rho')$  of  $\cal U(\cal C)$ we {\em define}  $\varphi_{I,J} \colon I \longrightarrow  J$ 
to be the following composite
\begin{equation}\label{phi}
\xymatrix{I\ar[r]^-{\rho'_{I}}& I\otimes J \ar[r]^-{\lambda_{J}} & J }
\end{equation} 
so with this notation $\varphi_{J,I} \colon J \longrightarrow  I$ is the following composite
\begin{align*} 
\xymatrix{J\ar[r]^-{\rho_{J}}& J\otimes I \ar[r]^-{\lambda'_{I}} & I }.
\end{align*}
When no confusion arises we will call these maps just $\varphi$.

\begin{lem}\label{compat}
The map $\varphi_{I,J}$ defined by \eqref{phi} is an arrow in $\cal U (\cal C)$ from $(I,\lambda,\rho)$ to $(J,\lambda',\rho')$.
\end{lem}
\begin{proof}
We show that the first diagram of \eqref{Comp} commutes by considering the following diagram 
\begin{equation*}
\begin{aligned}
\xymatrix{
I\otimes X \ar[rr]^{\rho'_{I}\otimes 1_ {X}} \ar[drdr]_{1_{X\otimes Y}} && (I\otimes J)\otimes X    
  \ar[rr]^{\lambda_J\otimes 1_{X}} 
\ar[d]_{\alpha_{I,J,X}}  && J\otimes X \ar[d]^ {\lambda'_{X}} 
\\ && I\otimes (J\otimes X) \ar[rur]_{\lambda_{J\otimes X}}
\ar[d]^{1_{I}\otimes\lambda'_{X}} && X
\\ && I\otimes X \ar[rur]_{\lambda_{X}} }
\end{aligned}
\end{equation*}
in which the left hand triangle commutes by equation \eqref{midunit} for $(J,\lambda',\rho')$, the right-hand triangle commutes by equation (\ref{leftunit}) for $(I,\lambda,\rho)$, and the rectangle commutes by the naturality of $\lambda$.
The right hand side of \eqref{Comp} is analogous. \qedhere
\end{proof}
\begin{prop}\label{unique}
There is exactly one morphism from $(I,\lambda,\rho)$ to $(J,\lambda',\rho')$ in $\cal U(\cal C)$.
\end{prop}
\begin{proof}
Suppose we have another  morphism $\tau$ from $I$ to $J$ in $\cal U(\cal C)$, and consider the following diagram
\begin{equation*}
\begin{aligned}
\xymatrix{
I \ar[r]^{\rho'_{I}} \ar[d]_{\tau}
& I\otimes J \ar[d]_{\tau\otimes 1_{J}} 
\ar[drr]^{\lambda_{J}}
\\ J\ar[r]_{\rho'_{J}} \ar@/_2pc/[rrr]_{1_J} & J\otimes J 
\ar[rr]_{\lambda'_{J}} 
&& J   }
\end{aligned} 
\end{equation*}
The square commutes by the naturality of $\rho'$, the triangle commutes by the assumption that $\tau$ satisfies the left hand side of equation \eqref{Comp}, and the semi-circle commutes by \eqref{unitunit} for $(J,\lambda',\rho')$.
This shows that $\tau = \varphi$. \qedhere
\end{proof}

\begin{cor}\label{iso}
Any two objects $(I,\lambda,\rho)$ and $(J,\lambda',\rho')$ in $\cal U(\cal C)$ are isomorphic.
\end{cor}
\begin{proof}
Both $\varphi_{J,I} \circ \varphi_{I,J}$ and $1_{I}$  are arrows from $(I,\lambda,\rho)$ to $(I,\lambda,\rho)$ in $\cal U(\cal C)$ so by uniqueness they are equal. That $\varphi_{I,J} \circ \varphi_{J,I} = 1_{J}$ is analogous. \qedhere
\end{proof}

Thus a skew semimonoidal category is a {\em skew monoidal category} if $\cal U(\cal C)$ is non-empty.

Proposition~\ref{unique} and Corollary~\ref{iso} then imply that the units for a skew monoidal category are unique up to a unique isomorphism (if they exist).
That is, the category  $ \cal U(\cal C)$ is equivalent to a terminal category.

Next, we shall see that either $\lambda$ or $\rho$ determines the other.
\begin{cor}\label{fixed}
If $(I,\lambda,\rho')$ and $(I,\lambda,\rho)$ are in $\cal U(\cal C)$ then $\rho'$ = $\rho$.
\end{cor}
\begin{proof}
Consider the unique morphism $\varphi_{J,I} \colon (I,\lambda,\rho') \longrightarrow (I,\lambda,\rho)$ where $J = (I,\lambda,\rho')$. By \eqref{unitunit}, this must be $1_{I}$; then by \eqref{Comp} we deduce that $\rho = \rho'$. \qedhere
\end{proof}
\begin{cor}
If $(I,\lambda,\rho)$ and $(I,\lambda',\rho)$ are in $\cal U(\cal C)$ then $\lambda$ = $\lambda'$.
\end{cor}
\begin{proof}
Dually, by reversing the tensor and the direction of arrows, we can instead repeat the above argument instead using $\varphi_{I,J}$. \qedhere
\end{proof}

\begin{remark}
Equation \eqref{pentagon} or the pentagon equation was not used in Proposition~\ref{unique} or its Corollaries. This leads to the possibility of similar results about the units of skew versions of categories not satisfing \eqref{pentagon} such as in \cite{J}.
\end{remark}
\begin{remark}
The proof of Lemma~\ref{compat} uses equations \eqref{leftunit},  \eqref{midunit} and \eqref{rightunit} but not \eqref{pentagon} or \eqref{unitunit}, while the proof of Proposition~\ref{unique} uses the equations \eqref{unitunit} and \eqref{Comp}. Now suppose that $\lambda$ and $\rho$ satisfy only \eqref{leftunit}, \eqref{midunit} and \eqref{rightunit}. Then the composite $\xymatrix{I \ar[r]^-{\rho_{I}} & I\otimes I \ar[r]^-{\lambda_{I}}
& I}$ satisfies \eqref{Comp} and so \eqref{unitunit} becomes a special case of the uniqueness result in Proposition~\ref{unique}.  
\end{remark} 

We denote a skew monoidal category by the 6-tuple $({\cal C},\otimes,I,\alpha,\lambda,\rho)$.
For the following proposition, we use the fact that, as $\otimes$ is a bifunctor, the interchange law holds, in particular, 
\begin{align*}
(f\otimes 1) \circ (1\otimes g) = (1\otimes g) \circ (f\otimes 1)
\end{align*}. 

\begin{prop}\label{closedunit}
Let $({\cal C},\otimes,I,\alpha,\lambda,\rho)$ be a skew monoidal category . If there exists an object $J$ and an isomorphism $\varphi \colon J \longrightarrow  I$ in $\cal C$ then $(J, \lambda', \rho')$ is also a unit of $\cal C$, where $\lambda'_{X} \colon J\otimes X \longrightarrow  X$ and $\rho'_{X} \colon X \longrightarrow  X\otimes J$ are given by the following composites:
\begin{equation*}
\vbox{\xymatrix{J\otimes X \ar[rr]^{\varphi \otimes 1_{X}}&& I\otimes X \ar[rr]^-{\lambda_{X}} && X }}
\qquad  \qquad
\vbox{\xymatrix{X \ar[rr]^-{\rho_{X}}&& X\otimes I \ar[rr]^{1_{X}\otimes \varphi^{-1}} && X\otimes J }}
\end{equation*}
\end{prop}
\begin{proof} 
We need to show that these composites satisfy the four conditions in the definition. Consider the following diagrams. For the one on the left, the square commutes by the naturality of $\alpha$ and the triangle commutes by \eqref{leftunit}. For the one on the right, the square commutes by the naturality of $\alpha$ and the triangle commutes by \eqref{rightunit}. 
\begin{align*}
\vbox{\xymatrix{
(J\otimes X)\otimes Y  \ar[rr]^{\alpha} \ar[dd]^{(\varphi\otimes 1)\otimes 1} 
&& J\otimes (X\otimes Y) \ar[dd]_{\varphi\otimes 1}
\\
\\
(I\otimes X)\otimes Y \ar[rr]^{\alpha} \ar[dr]_{\lambda\otimes 1} && I\otimes (X\otimes Y) \ar[ld]^{\lambda}
\\ & X\otimes Y }}
\qquad 
\vbox{\xymatrix{
(X\otimes Y)\otimes J  \ar[rr]^{\alpha} 
&& X\otimes (Y\otimes J) 
\\
\\
(X\otimes Y)\otimes I \ar[rr]^{\alpha} \ar[uu]_{1\otimes {\varphi^{-1}}} && X\otimes (Y\otimes I) \ar[uu]^{1\otimes (1\otimes {\varphi^{-1}})}
\\ & X\otimes Y \ar[ur]_{1\otimes \rho} \ar[ul]^{\rho}
 }}
\end{align*}
For the following diagram, the square commutes by the naturality of $\alpha$ and the outside commutes by \eqref{midunit}. The semicircle on the left commutes as $\varphi$ is an isomorphism. Thus the irregular lower region commutes as required.
\begin{align*}
\xymatrix{
(X\otimes I)\otimes Y  \ar[rr]^{\alpha} 
&& X\otimes (I\otimes Y) \ar@/^3.5pc/[dddd]^{1\otimes \lambda}
\\
\\
(X\otimes J)\otimes Y \ar[rr]^{\alpha} \ar[uu]_{(1\otimes \varphi)\otimes 1} && X\otimes (J\otimes Y) \ar[uu]^{1\otimes (\varphi\otimes 1)}
\\
\\  (X\otimes I)\otimes Y \ar[uu]_{(1\otimes \varphi^{-1})\otimes 1} \ar@/^3.5pc/[uuuu]^{(1\otimes 1)\otimes 1} & X\otimes Y \ar[r]_1 \ar[l]_-{\rho \otimes 1} & X\otimes Y
}
\end{align*}
For the final diagram, the top right square commutes by the interchange law, the bottom right square commutes by the naturality of $\lambda$, the top left square commutes by the naturality of $\rho$ and below this, the upper triangle commutes by \eqref{unitunit}.
\begin{align*}
\xymatrix{
J \ar[rr]^{\rho_{J}} \ar[dd]_{\varphi} \ar@/_2.6pc/[dddd]_1 && J\otimes I \ar[rr]^{1\otimes {\varphi^{-1}}} \ar[dd]^{\varphi\otimes 1} && J\otimes J \ar[dd]^{\varphi\otimes 1}
\\
\\ I \ar[rr]^{\rho_{I}} \ar[dd]_{\varphi^{-1}} \ar[ddrr]_1 && I\otimes I \ar[rr]^{1\otimes {\varphi^{-1}}} \ar[dd]^{\lambda_{I}} && I\otimes J \ar[dd]^{\lambda_{J}}
\\
\\ J \ar[rr]_{\varphi} \ar@/_2.4pc/[rrrr]_1 && I \ar[rr]_{\varphi^{-1}} && J }
\end{align*}
\end{proof}

\begin{prop}\label{Iso}
If $({\cal C},\otimes,I,\alpha,\lambda,\rho)$ is a skew monoidal category then 

$(I\otimes I, \lambda_{X}\circ (\lambda_{I}\otimes 1_{X}), (1_{X}\otimes \rho_{I})\circ \rho_{X})$ is a unit of $\cal C$ if and only if $\lambda_{I} \colon I\otimes I \longrightarrow  I$ is invertible.
\end{prop}
\begin{proof}
Assuming that $(I\otimes I, \lambda_{X}\circ (\lambda_{I}\otimes 1_{X}), (1_{X}\otimes \rho_{I})\circ \rho_{X})$ is a unit we can use Lemma~\ref{compat} and Proposition~\ref{unique} and just show that $\varphi_{I\otimes I,I} = \lambda_{I}$ by considering the following diagram.
\begin{align*}
\xymatrix{I\otimes I \ar[dd]_{\lambda_I}   
\ar[rr]^{\rho_{I\otimes I}} && (I\otimes I)\otimes I \ar[dd]^{\lambda_{I}\otimes 1_{I}}
\\
\\ I \ar[rr]^{\rho_{I}} \ar[dr]_{1_{I}} && I\otimes I \ar[dl]^{\lambda_I}
\\ & I
}
\end{align*}
where the square commutes by the naturality of $\rho$ and the triangle commutes by \eqref{unitunit}.

Conversely, if $\lambda_{I} \colon I\otimes I \longrightarrow  I$ is invertible then by \eqref{unitunit}  $\rho_{I} \circ \lambda_{I} = 1_{I\otimes I}$, so then $\lambda_{I}^{-1} = \rho_{I}$ and we can apply Proposition~\ref{closedunit}.\qedhere
\end{proof}

A skew monoidal category $({\cal C},\otimes, I,\alpha,\lambda,\rho)$ is {\em weakly normal} if it also satifies the condition that $\rho_{I} \circ \lambda_{I} = 1_{I\otimes I}$; equivalently, if $\lambda_{I} \colon I\otimes I \longrightarrow I$ is invertible.

\begin{prop}\label{Nat}
If $({\cal C},\otimes,I,\alpha,\lambda,\rho)$ is a weakly normal skew monoidal category then the monoid $\End(I)$ of endomorphisms of the unit object $I$ is commutative.
\end{prop}
\begin{proof} 
As $\lambda_{I} \colon I\otimes I \longrightarrow  I$ is invertible, it induces an isomorphism $\psi \colon \End(I\otimes I) \longrightarrow \End(I)$ defined by $\psi(\gamma) = \lambda_{I} \circ \gamma \circ \lambda_{I}^{-1}$. For $f \in \End(I)$ we deduce, by the naturality of $\lambda$, that
\begin{eqnarray*}
f & = & f \circ \lambda_{I} \circ \lambda_{I}^{-1} \\
 & = & \lambda_{I} \circ (1_{I}\otimes f) \circ \lambda_{I}^{-1} \\
 & = & \psi(1_{I}\otimes f)
\end{eqnarray*}
Similarly, using the naturality of $\lambda^{-1}$ we get $f = \psi(f\otimes 1_{I})$.

So for $f, g \in \End(I)$ we have, by the interchange law, that
\begin{eqnarray*}
f\circ g & = & \psi(f \otimes 1)\circ \psi(1 \otimes g) \\
 & = & \psi((f\otimes 1)\circ (1\otimes g)) \\
 & = & \psi((1\otimes g)\circ (f\otimes 1)) \\
 & = & \psi(1\otimes g) \circ \psi(f\otimes 1) \\
 & = & g\circ f
\end{eqnarray*} \qedhere
\end{proof}
\begin{remark}
Let $R$-$\textbf{Mod}$ denote the category of left $R$-modules over some ring $R$. Regarding $R$ as a left module over itself using its product, and noticing that $\End(R)$ is the monoid $R$ if we regard $R$ as a monoid under multiplication, we can use Lemma~\ref{Nat} to conclude that if $R$ is a non-commutative ring then $R$ is not the unit object for a weakly normal skew monoidal structure on $R$-$\textbf{Mod}$.
\end{remark}

A skew monoidal category is {\em left normal} if $\lambda$ is invertible. This implies that tensoring on the left by $I$ is an equivalence. Using the naturality of $\lambda$ and the invertibilty of $\lambda_{X}$ we deduce that
\begin{align*}
\lambda_{I \otimes X} & = 1_{I} \otimes {\lambda_{X}}
\end{align*}
A skew monoidal category is {\em right normal} if $\rho$ is invertible and {\em normal} if both $\lambda$ and $\rho$ are invertible.

\begin{remark}
If $\cal C$ is a left normal skew monoidal category then for any units $I$ and $J$ in $\cal C$ we have $I \otimes J \cong J$ and so $I \otimes J$ is also a unit by Proposition~\ref{Iso}. Thus, if $\cal C$ is a left normal skew monoidal category then the $\otimes$ from $\cal C$ applied to $\cal U(\cal C)$ gives $\cal U(\cal C)$ the structure of a skew semimonoidal category.
\end{remark}

\begin{lem}
If $({\cal C},\otimes,I,\alpha,\lambda,\rho)$ only satisfies \eqref{leftunit} and \eqref{midunit} with both $\lambda$ and $\rho$ being invertible then \eqref{unitunit} holds.
\end{lem}
\begin{proof} 
Consider the following diagram:
\begin{align*}
\xymatrix{ (I\otimes I) \otimes X \ar[rr]^{\alpha} \ar[ddrr]^{\lambda_{I}\otimes 1_{X}} && I\otimes (I\otimes X) \ar[dd]^{1_{I}\otimes {\lambda_{X}} = {\lambda_{I\otimes X}}}
\\
\\ I\otimes X \ar[uu]^{\rho_{I}\otimes 1_{X}} \ar[rr]_{1} && I\otimes X}
\end{align*}
The outside commutes by \eqref{midunit}, the upper triangle commutes by \eqref{leftunit} where we used the assumption of $\lambda$ being invertible and the resulting identity that $1_{I}\otimes {\lambda_{X}} = {\lambda_{I\otimes X}}$, so then the lower triangle commutes. Now taking $X = I$ and using the assumption that $\rho$ is a natural isomorphism we get \eqref{unitunit}. \qedhere
\end{proof}

\subsection{Independence of the Axioms}
In this section we show that the five axioms for a skew monoidal category, given by equations \eqref{pentagon},\eqref{leftunit},\eqref{midunit},\eqref{rightunit} and \eqref{unitunit}, are independent. The underlying category we use is \textbf{Set} where the cartesian product between two sets is denoted by $\times$; we often identify the cartesian product of a one-point set with a set as the set itself and $X \times Y$ with $Y \times X$ in what follows.

For a set $M$, define a tensor product on \textbf{Set} by $X \otimes Y = M \times X \times Y$; this gives a functor $\xymatrix{\textbf{Set} \times \textbf{Set} \ar[r]^-{\otimes} & \textbf{Set}}$. If $M$ has a product $\xymatrix{M \times M \ar[r]^-{.} & M }$ there is a natural transformation $\alpha \colon M \times M \times X \times Y \times Z \longrightarrow M \times X \times M \times Y \times Z$ given by sending $(m,n,x,y,z)$ to $(m.n,x,m,y,z)$. Let $I$ be a one-point set, and $1 \in M$. The map $\lambda \colon I\otimes X (= M \times X) \longrightarrow X$ defined by sending $(m,x)$ to $x$ and the map $\rho \colon X \longrightarrow X \otimes I (=M \times X)$ defined by sending $x$ to $(x,1)$ are both natural transformations. With these maps, equation \eqref{pentagon} asks that the product on $M$ is associative, equation \eqref{midunit} asks that $1 \in M$ is a right identity on $M$, and equation \eqref{rightunit} asks that $1 \in M$ is a left identity on $M$. The remaining two equations are already satisfied under these maps and impose no extra structure on the set $M$. (These maps are based on the constructions in the first section of \cite{SkMon}.)

We take for $M$ the following three sets. The $M$ defined by the table on the left has a left and right identity but is not associative, so equation \eqref{pentagon} does not hold but the other four equations do.
\begin{align*}  
\fbox{\begin{tabular}{l|rlr} 
. & 1 & a & b \\ \hline
1 & 1 & a & b \\
a & a & 1 & b \\
b & b & a & 1 \\
\end{tabular}} &&&
\fbox{\begin{tabular}{l|rl}  
. & 1 & a  \\ \hline
1 & 1 & a  \\
a & 1 & a  \\
\end{tabular}} &
\fbox{\begin{tabular}{l|rl} 
. & 1 & a  \\ \hline
1 & 1 & 1  \\
a & a & a  \\
\end{tabular}}
\end{align*}
The $M$ defined by the table in the middle has no right identity, but has a left identity and is associative. In this case, equation \eqref{midunit} does not hold but the other four equations do. The $M$ defined by the table on the right has no left identity, but has a right identity and is associative. In this case, the equation \eqref{rightunit} does not hold but the other four equations do.

Thus each of the equations \eqref{pentagon}, \eqref{midunit} and \eqref{rightunit} is independent of the remaining four equations.
By reversing the tensor, direction of arrows and the order of composition we notice that equation \eqref{leftunit} and equation \eqref{rightunit} are dual, so statements such as independence holds for one if and only if it holds for the other. Thus independence of equation \eqref{leftunit} follows from the independence of equation \eqref{rightunit}. 

This leaves equation \eqref{unitunit}, for which we take the tensor product to be the cartesian product, so $X \otimes Y = X \times Y$.
The map $\alpha \colon (X \times Y) \times Z \longrightarrow  X \times (Y \times Z)$ is the usual associative isomorphism $(X \times Y) \times Z \cong  X \times (Y \times Z)$ and $I$ is given by $\{ a,b\}$. The map $\lambda \colon I \times X \longrightarrow X$ defined by sending $(i,x)$ to $x$ and the map $\rho \colon X \longrightarrow X \times I$ defined by sending $x$ to $(x,a)$ are natural transformations. In this case, equation \eqref{unitunit} asks for the elements of $I$ to be identical which is not the case here, so equation \eqref{unitunit} is not satisfied but it is easy to see that the other four equations hold.

With these four examples and duality we have shown that:
\begin{prop}
The five axioms for a skew monoidal category are independent.
\end{prop}

\subsection{Monoidal Functors}

Let $({\cal C},\otimes',I,\alpha',\lambda',\rho')$ and $({\cal D},\otimes,J,\alpha,\lambda,\rho)$ be skew monoidal categories. A {\em monoidal functor} from $\cal C$ to $\cal D$ is a triple $(F,\varphi,F_{0})$ where $F\colon \cal C \longrightarrow \cal D$ is a functor of the underlying categories, $F_{0}$ is a morphism $J \longrightarrow F(I)$ in $\cal D$ and $\varphi$ is a natural transformation with components $\varphi_{X,Y}\colon F(X)\otimes F(Y) \longrightarrow F(X\otimes' Y)$ such that the following diagrams commute.
\begin{align}\label{funas}
\xymatrix{ (F(X)\otimes F(Y))\otimes F(Z) \ar[rrr]^{\alpha_{F(X),F(Y),F(Z)}} \ar[d]_{\varphi_{X,Y}\otimes 1_{F(Z)}} &&& F(X)\otimes (F(Y)\otimes F(Z)) \ar[d]^{1_{F(X)}\otimes \varphi_{Y,Z}}
\\ F(X\otimes' Y)\otimes F(Z) \ar[d]_{\varphi_{X\otimes' Y,Z}} &&& F(X)\otimes F(Y\otimes' Z) \ar[d]^{\varphi_{X,Y\otimes' Z}}
\\ F((X\otimes' Y)\otimes' Z) \ar[rrr]_{F(\alpha'_{X,Y,Z})} &&& F(X\otimes' (Y\otimes' Z))}
\end{align} 
\begin{equation}\label{FunUn} 
\begin{aligned}
\vbox{\xymatrix{
J\otimes F(X) \ar[rr]^{\lambda_{F(X)}} \ar[dd]_{F_{0}\otimes 1_{F(X)}} && F(X)
\\
\\ F(I)\otimes F(X) \ar[rr]_{\varphi_{I,X}} \ar@{} [rruu]| {} && F(I\otimes' X) \ar[uu]_{F(\lambda'_{X})} }}
  \qquad
\vbox{\xymatrix{
F(X)  \ar[dd]_{F(\rho'_{X})} \ar[rr]^{\rho_{F(X)}} && F(X)\otimes J \ar[dd]^{1_{F(X)\otimes F_{0}}}
\\
\\ F(X\otimes' I) \ar@{} [rruu]| {} && F(X)\otimes F(I) \ar[ll]^{\varphi_{X,I}} }}
\end{aligned}
\end{equation}
A monoidal functor between skew monoidal categories is {\em normal} if $F_{0}$ is an isomorphism, and is {\em strong} if both $\varphi$ and $F_{0}$ are isomorphisms. If the skew monoidal categories were monoidal then these are the usual notions of lax, normal, and strong monoidal functors.

\begin{prop}\label{UniqueFun}
Let $({\cal C},\otimes',I,\alpha',\lambda',\rho')$ and $({\cal D},\otimes,J,\alpha,\lambda,\rho)$ be skew monoidal categories and let $F$ be a functor and $\varphi$ a natural transformation such that \eqref{funas} holds. Then there is at most one $F_{0}$ such that \eqref{FunUn} holds.
\end{prop}
\begin{proof}
Let $F_{0}^{*}$ be another such morphism in $\cal D$, so in particular $F_{0}^{*} \colon J \longrightarrow F(I)$ satisfies the equations in \eqref{FunUn}. Consider the following diagram.
\begin{equation*}
\xymatrix{J \ar[rr]^{\rho_{J}} \ar[dd]_{F^{*}_{0}} \ar@/^3pc/[rrrr]^{1_{J}} && J\otimes J \ar[rr]^{\lambda_{J}} \ar[dd]_{F_{0}^{*}\otimes 1_{J}} \ar[ddddrr]^{1_{J}\otimes F_{0}} && J \ar@/^3pc/[dddddd]^{F_{0}}
\\
\\ F(I) \ar[rr]^{\rho_{F(I)}} \ar@/_3.8pc/[dddd]_{1_{F(I)}} \ar[dd]_{F(\rho_{I})} && F(I)\otimes J \ar[dd]_{1_{F(I)}\otimes F_{0}}
\\
\\ F(I\otimes' I) \ar[ddrrrr]^{F(\lambda_{I})} && F(I)\otimes F(I) \ar[ll]_{\varphi_{I,I}} && J\otimes F(I) \ar[ll]_{F_{0}^{*}\otimes {1_{F(I)}}} \ar[dd]_{\lambda_{F(I)}}
\\
\\ F(I) \ar[rrrr]_{1_{F(I)}} &&&& F(I)
}
\end{equation*}
The part involving the semicircles on the top and left hand side commute by \eqref{unitunit}. The top square commutes by the naturality of $\rho$, and the square below it commutes by the right hand equation in \eqref{FunUn}. The triangle next to the squares commutes by the interchange law. The bottom triangle commutes by the left hand equation in \eqref{FunUn} and the remaining part of the diagram (on the right) commutes by the naturality of $\lambda$. The commutativity of the exterior gives the required uniqueness. \qedhere
\end{proof}

\begin{remark}
If $F_{0}\colon J \longrightarrow F(I)$ is an isomorphism in $\cal D$ then by Proposition~\ref{closedunit}, $F(I)$ is also a unit in $\cal D$. Now, by Proposition~\ref{unique}, there is a unique morphism between these units, namely $\varphi_{J,F(I)}$ and using the naturality of $\lambda$ it can be shown that $\varphi_{J,F(I)} = F_{0}$. 
\end{remark}

\begin{remark} 
This lemma generalises the uniqueness result of Proposition~\ref{unique}, which we may recover on taking the two skew monoidal categories to be the same, $F$ to be the identity functor, and $\varphi$ the identity natural transformation. It also implies the uniqueness of units for monoids in a skew monoidal category by taking $\cal C$ $= 1$.
\end{remark}

\begin{remark}
We denote by \textbf{SkMon} the category with {\em objects}  skew monoidal categories and {\em 1-cells} monoidal functors and \textbf{SkSemiMon} the category with {\em objects} skew semimonoidal categories and {\em 1-cells} semimonoidal functors (drop the $F_0$ conditions for the unit).

We denote the obvious forgetful functor where we drop all reference to units and any associated conditions by
 $V \colon \textbf{SkMon} \longrightarrow \textbf{SkSemiMon}$ .
For an object $\cal C$ of \textbf{SkSemiMon}, that is, $\cal C$ is a skew semimonoidal category, the fibre of $V$ at $\cal C$ is the category $\cal U(\cal C)$ of $\cal C$.

The uniqueness of $F_0$ in Proposition~\ref{UniqueFun} implies that the forgetful functor $V$ is faithful. Moreover, the uniqueness and existence results from Section 2.2 imply that $V$ is also full on isomorphisms in \textbf{SkSemiMon}, and by Proposition~\ref{closedunit} $V$ is also an isofibration. 
\end{remark}

\section{Skew Monoidales}

The results of the previous section can be lifted to skew monoidales, these were first defined in \cite{SkMon} as an enriched version of a skew monoidal category. So in this section we internalise the main result of the previous section. By the coherence results of \cite{GPS}, however, it will suffice to work in a Gray monoid.  
 
Let $\cal B$ be a Gray monoid; see \cite{Day} for an explicit definition. Note that in a Gray monoid, for the 1-cells $f : A \longrightarrow A'$ and $g : B \longrightarrow B'$ , the only structural 2-cells are the invertible 2-cells of the form 
\begin{align*}
\xymatrix {
A\otimes B \ar[rr]^{1\otimes g} \ar[d]_{f \otimes 1} \ar@{}[drr]|{\cong} && A\otimes B' \ar[d]^{f\otimes 1}
\\ A'\otimes B \ar[rr]_{1\otimes g} && A'\otimes B'}
\end{align*} 
or tensors and composites thereof.
In this section we denote them with the symbol $\cong$ as above. These 2-cells satisfy four axioms which we do not list but will appeal to throughout the rest of this section; see \cite{Day} once again. We write $I$ for the unit object of the Gray monoid.
 
A {\em skew semimonoidal} structure on an object $A$ in $\cal B$ 
consists of a morphism $p \colon A\otimes A \longrightarrow A$ called the {\em tensor product}, and a 2-cell 

\begin{equation*}\label{smonoidale1}
\xymatrix{
 A\otimes A\otimes A \ar[rr]^{1\otimes p}_{~}="1" \ar[d]_{p\otimes 1}  && A\otimes A \ar[d]^p
  \\  A\otimes A \ar[rr]_p^{~}="2"
  \ar@{=>}"2";"1"^\alpha && A }
 \end{equation*}
subject to the "pentagon" axiom
\begin{equation}\label{mpent} 
\vbox {\xymatrix@!=0.9pc{
&& A\otimes A\otimes A \ar[dd]_{1\otimes p}^{~}="2" \ar[drr]^{p\otimes 1}
\\ A\otimes A\otimes A\otimes A \ar[urr]^{p\otimes 1\otimes 1} \ar[dd]_{1\otimes 1\otimes p} \ar@{}[drr]|{\cong} &&&&
A\otimes A \ar[dd]^p_{~}="1" 
\\ && A\otimes A\ar[drr]^p_{~}="3" 
\\ A\otimes A\otimes A \ar[drr]_{1\otimes p}^{~}="4" 
\ar[urr]^{p\otimes 1} &&&& A
\\ && A\otimes A \ar[urr]_p
\ar@{=>}"1";"2"_{\alpha}
\ar@{=>}"3";"4"_{\alpha}
  }}
  = 
\vbox {\xymatrix@!=0.9pc{
&& A\otimes A\otimes A \ar[drr]^{p\otimes 1}_{~}="3"
\\ A\otimes A\otimes A\otimes A \ar[urr]^{p\otimes 1\otimes 1} \ar[dd]_{1\otimes 1\otimes p}^{~}="4" 
\ar[drr]_{1\otimes p\otimes 1}^{~}="1" &&&&
A\otimes A \ar[dd]^p_{~}="5" 
\\ && A\otimes A\otimes A \ar[dd]^{1\otimes p}_{~}="2" \ar[urr]_{p\otimes 1}
\\ A\otimes A\otimes A \ar[drr]_{1\otimes p} &&&& A
\\ && A\otimes A \ar[urr]_p_{~}="6"
\ar@{=>}"3";"1"_{{\alpha}\otimes 1}
\ar@{=>}"2";"4"_{1\otimes {\alpha}}
\ar@{=>}"5";"6"+/va(145)+3.62pc/_{\alpha}}}  
\end{equation}

An object $A$ of $\cal B$ equipped with such a skew semimonoidal structure is called a {\em skew semimonoidale} in $\cal B$;
we denote it by $(A,p,\alpha)$.

A skew semimonoidale in the cartesian monoidal 2-category \textbf{Cat} of categories, functors and natural transformations is a skew semimonoidal category.

If $(A,p,\alpha)$ is a skew semimonoidale in $\cal  B$, we form a category $\cal U$$(A,p,\alpha)$ as follows. The {\em objects} are triples $(j,\lambda,\rho)$, called {\em units}, where $j$ is a morphism $j \colon I \longrightarrow A$ in $\cal B$  equipped with 2-cells, denoted by $\lambda$ and $\rho$,  called the {\em left unit} and {\em right unit constraints}. These have the form 
 
$$\vbox{\xymatrix {
A\otimes A \ar[ddrr]_p && A \ar[ll]_{j\otimes 1} \ar[dd]^1_{~}="1"
\\
\\ && A
\ar@{=>}"1"+/va(-199)+3pc/;"1"_{\lambda}
}}
\quad   \quad
\vbox{\xymatrix {
&& A\ar[dd]^{1\otimes j}
\\
\\ A \ar[rruu]^1_{~}="1" && A\otimes A \ar[ll]^p
\ar@{<=}"1"-/va(-199)+3pc/;"1"_{\rho}
}}$$
and are required to satisfy the following four equations
\begin{equation}\label{mleft}
\vbox{\xymatrix {
A\otimes A \ar[d]_{1\otimes p} \ar[r]^{j\otimes 1\otimes1} \ar@{}[dr]|{\cong}
 & A\otimes A\otimes A \ar[d]_{1\otimes p}^{~}="2" \ar[r]^{p\otimes 1}  & A\otimes A \ar[d]^p_{~}="1"
\\ A \ar[r]^{j\otimes 1} \ar@/_2pc/[rr]_1^{~}="3" & A\otimes A \ar[r]^p & A
\ar@{=>}"1";"2"_{\alpha}
\ar@{<=}"3";"3"+/va(90)+1.1pc/^{\lambda} 
}} 
 =
\vbox{\xymatrix {
A\otimes A \ar[d]_{1\otimes p} \ar[r]^{j\otimes 1\otimes1} \ar@/_1.7pc/[rr]_1^{~}="1"
 & A\otimes A\otimes A  \ar[r]^{p\otimes 1}  & A\otimes A \ar[d]^p
\\ A \ar[rr]_1 \ar@{}[rr]^{\cong} && A
\ar@{<=}"1";"1"+/va(90)+1.1pc/^{\lambda\otimes 1}
}}
\end{equation}
\begin{equation}\label{mright}
\vbox{\xymatrix {
A\otimes A \ar[d]_{p\otimes 1} \ar[r]^{1\otimes 1\otimes j} \ar@{}[dr]|{\cong}
 & A\otimes A\otimes A \ar[d]_{p\otimes 1} \ar[r]^{1\otimes p}_{~}="2"  & A\otimes A \ar[d]^p
\\ A \ar[r]_{1\otimes j} \ar@/_2pc/[rr]_1^{~}="3" & A\otimes A \ar[r]_p^{~}="1" & A
\ar@{=>}"1";"2"_{\alpha}
\ar@{=>}"3";"3"+/va(90)+1.1pc/^{\rho} 
}} 
 = 
\vbox{\xymatrix {
A\otimes A \ar[d]_{p\otimes 1} \ar[r]^{1\otimes 1\otimes j} \ar@/_1.7pc/[rr]_1^{~}="1"
 & A\otimes A\otimes A  \ar[r]^{1\otimes p}  & A\otimes A \ar[d]^p
\\ A \ar[rr]_1 \ar@{}[rr]^{\cong} && A
\ar@{=>}"1";"1"+/va(90)+1.1pc/^{1\otimes {\rho}}
}}
\end{equation}
\begin{equation}\label{mmid}
\vbox{\xymatrix {
A\otimes A \ar[rr]^1_{~}="1" \ar[dr]_{1\otimes j\otimes 1} \ar@/_1.6pc/[ddr]_1^{~}="2" && A\otimes A \ar[dd]^p_{~}="3"
\\ & A\otimes A\otimes A \ar[d]_{1\otimes p}^{~}="4" \ar[ru]_{p\otimes 1}
\\ & A\otimes A \ar[r]_p & A
\ar@{=>}"1";"1"+/va(-90)+1.3pc/^{{\rho}\otimes 1}
\ar@{<=}"2";"2"+/va(-315)+1.3pc/_{1\otimes {\lambda}}
\ar@{=>}"3";"4"^{\alpha}
}} 
\qquad = \qquad 
\vbox{\xymatrix {
A\otimes A \ar[rr]^1 \ar[dd]_1 \ar@{}[ddrr]|{\equiv} && A\otimes A \ar[dd]^p
\\
\\ A\otimes A \ar[rr]_p && A
}} 
\end{equation}
\begin{equation}\label{muu} 
\vbox{\xymatrix {
I \ar[rr]^j \ar[d]_j \ar@{}[drr]|{\cong} && C \ar[d]^{1\otimes j} \ar@/^0.9pc/[drr]^1_{~}="1"
\\ C \ar[rr]_{j\otimes 1} \ar@/_1.7pc/[rrrr]_1_{~}="2" && C\otimes C \ar[rr]_p && C
\ar@{=>}"1";"1"+/va(-135)+1.3pc/^{\rho}
\ar@{<=}"2";"2"+/va(90)+1.3pc/^{\lambda}
}} 
\qquad =\qquad 
\vbox{\xymatrix {
I \ar[rr]^j \ar[d]_j \ar@{}[drr]|{\equiv} && C \ar[d]^1
\\ C \ar[rr]_1 && C
}}
\end{equation}

The {\em arrows} of $\cal U$$(A,p,\alpha)$ from $(j',\lambda',\rho')$ to $(j,\lambda,\rho)$ are given by the 2-cells $\xymatrix{j' \ar[r]^{\varphi} & j}$
 in $\cal B$ satisfying the following equations
\begin{equation}\label{sComp}
\vbox{\xymatrix {p(j' \otimes 1) \ar[d]_{p(\varphi \otimes 1)} \ar[dr]^{\lambda'}
\\ p(j \otimes 1) \ar[r]_(.6){\lambda} & 1
}}
\qquad \qquad
\vbox{\xymatrix{ & p(1 \otimes j') \ar[d]^{p(1 \otimes {\varphi})}
\\ 1 \ar[r]_(.4){\rho} \ar[ru]^{\rho'} & p(1 \otimes j)
}} 
\end{equation}
In the case of a skew semimonoidal category seen as a skew semimonoidale $(A,p,j)$ for \textbf{Cat}, this agrees with the previous definition.
 
Given $(j',\lambda',\rho')$ and $(j,\lambda,\rho)$ in $\cal U$$(A,p,\alpha)$
we denote the following 2-cell by $\varphi_{j',j}$.
\begin{equation*}
\xymatrix{
&& I \ar[lld]_{j'} \ar[rrd]^j \ar@{}[dd]|{\cong} \\ A\ar[drr]^{1\otimes j} \ar[dddrr]_1^{~}="1"
&&&& A\ar[dll]_{{j'}\otimes 1}
\ar[dddll]^1_{~}="2" 
\\ && A\otimes A\ar[dd]_p_{~}="3"
\ar@{=>} ;"2"_{\lambda'}
\ar@{=>}"1";_{\rho}
\\ 
\\ && A }
\end{equation*}
When no confusion will arise we drop the subscripts and simply write $\varphi$.

\begin{lem}
The 2-cell $\varphi$ is an arrow in $\cal U$$(A,p,\alpha)$.
\end{lem}
\begin{proof}
We need to show that $\varphi$ satisfies \eqref{sComp}. We shall only verify the equation involving $\rho$; the one involving $\lambda$ is similar.

The composite $ \xymatrix { 1 \ar[rr]^{\rho'} && p(1 \otimes j') \ar[rr]^{p(1 \otimes {\varphi})} &&
 p(1 \otimes j)}$ appearing in \eqref{sComp} may be constructed as the following pasting composite.
 
\begin{equation*}
\xymatrix@!=0.9pc {
&&&& A \ar[lld]^{1\otimes {j'}}_{~}="5" \ar[rrd]_{1\otimes j}\ar@/_1.1pc/[lllld]_1^{~}="4" \ar@{}[dd]|{\cong}
\\ A \ar@/_1.1pc/[ddddrrrr]_1
&& A \otimes A\ar[drr]^{1\otimes 1\otimes j} \ar[dddrr]_1^{~}="1" \ar[ll]^p \ar@{}[ddd]|{\cong}
&&&& A\otimes A\ar[dll]_{1\otimes {j'}\otimes 1}
\ar[dddll]^1_{~}="2" 
\\ &&&& A\otimes A\otimes A\ar[dd]_{1\otimes p}_{~}="3"
\ar@{=>} ;"2"_{1\otimes {\lambda'}}
\ar@{=>}"1";_{1\otimes {\rho}}
\ar@{=>}"4";"4"+/va(-75)+1.5pc/_{\rho'}
\\ 
\\ &&&& A\otimes A \ar[d]^p
\\ &&&& A}
\end{equation*}
Using equation \eqref{mright} this is equal to  
\begin{equation*}   
\xymatrix@!=0.9pc {
&&&& A \ar[lld]^{1\otimes {j'}}_{~}="5" \ar[rrd]_{1\otimes j}\ar@/_1.1pc/[lllld]_1^{~}="4" \ar@{}[dd]|{\cong}
\\ A \ar@/_1.1pc/[ddddrrrr]_1^{~}="6" \ar[drr]_{1\otimes{j}}
&& A \otimes A\ar[drr]^{1\otimes 1\otimes j}  \ar[ll]^p \ar@{}[d]|{\cong}
&&&& A\otimes A\ar[dll]_{1\otimes {j'}\otimes 1}
\ar[dddll]^1_{~}="2" 
\\ && A\otimes A \ar[dddrr]_p^{~}="5" && A\otimes A\otimes A\ar[dd]^{1\otimes p}_{~}="3" \ar[ll]_{p\otimes 1}
\ar@{=>} ;"2"_{1\otimes {\lambda'}}
\ar@{=>}"4";"4"+/va(-75)+1.5pc/_{\rho'}
\ar@{=>}"5";"3"_\alpha
\ar@{=>}"6";"6"+/va(45)+1.6pc/_\rho
\\ 
\\ &&&& A\otimes A \ar[d]^p
\\ &&&& A}
\end{equation*} 
which, by properties of Gray monoids, is equal to 
\begin{equation*}
\xymatrix@!=0.8pc {
&&&& A \ar[rrd]^{1\otimes j}\ar@/_1.1pc/[lllld]_1^{~}="4" 
\\ A \ar@/_1.1pc/[ddddrrrr]_1^{~}="6" \ar[drr]^{1\otimes{j}} \ar@{}[rrrrr]|{\cong}
&& 
&&&& A\otimes A\ar[dll]_{1\otimes {j'}\otimes 1}\ar@/_1.4pc/[lllld]_1^{~}="7"
\ar[dddll]^1_{~}="2" 
\\ && A\otimes A \ar[dddrr]_p^{~}="5" && A\otimes A\otimes A\ar[dd]^{1\otimes p}_{~}="3" \ar[ll]_{p\otimes 1}
\ar@{=>} ;"2"_{1\otimes {\lambda'}}
\ar@{=>}"5";"3"_\alpha
\ar@{=>}"6";"6"+/va(45)+1.6pc/_\rho
\ar@{=>}"7";"7"+/va(-59)+1.9pc/_{\rho'\otimes 1}
\\ 
\\ &&&& A\otimes A \ar[d]^p
\\ &&&& A}
\end{equation*}
which finally, by equation \eqref{mmid}, is equal to

\begin{equation*} 
\xymatrix@!=0.6pc {
&&&& A \ar[rrd]^{1\otimes j}\ar@/_1.1pc/[lllld]_1^{~}="4" 
\\ A \ar@/_1.1pc/[ddddrrrr]_1^{~}="6" \ar[drr]^{1\otimes{j}} \ar@{}[rrrr]|{\cong}
&& 
&&&& A\otimes A \ar@/_1.4pc/[lllld]_1^{~}="7"
\ar[dddll]^1_{~}="2" 
\\ && A\otimes A \ar[dddrr]^p^{~}="5" 
\ar@{=>}"6";"6"+/va(45)+1.6pc/_{\rho} 
\\ 
\\ &&&& A\otimes A \ar[d]^p 
\\ &&&& A \ar@{}[luuuuu]|{\equiv}}
\end{equation*} \qedhere
\end{proof}
\begin{prop}
There is exactly one morphism from $(j,\lambda,\rho)$ to $(j',\lambda',\rho')$ in $\cal U$$(A,p,\alpha)$.
\end{prop}

\begin{proof} 
Let $\tau$ be another such 2-cell  $\xymatrix {
I\ar@/^1.4pc/[rr]^j_{~}="1" \ar@/_1.4pc/[rr]_{j'} && A
\ar@{=>}"1";"1"+/va(-85)+1.9pc/_{\tau}}$ in $\cal U$$(A,p,\alpha)$, so, in particular it satisfies 
\begin{equation}\label{tcom}  
\xymatrix { A \ar[dd]_{j\otimes 1}^{~}="1" \ar[ddrr]^1 \\
\\  A\otimes A \ar[rr]_p && A
\ar@{=>}"1";"1"+/va(9)+1.8pc/_{\lambda}}
=
\xymatrix { A \ar@/_1.5pc/[dd]_{j\otimes 1}^{~}="1"
\ar@/^1.5pc/[dd]^{j'\otimes 1} 
\ar@/^1.0pc/[ddrr]^1_{~}="2"
\\ 
\\  A\otimes A \ar[rr]_p^{~}="3" && A
\ar@{=>}"1";"1"+ /va(0)+1.8pc/_{\tau\otimes 1}
\ar@{<=}"2";"3"-/va(-60)+0.8pc/^{\lambda'}
}
\end{equation}
\hspace{.4in}

By assumption \eqref{tcom}, the 2-cell $\phi$
\begin{equation*} 
\xymatrix@!=0.6pc{
&& I \ar[lld]_j \ar[rrd]^{j'} \ar@{}[dd]|{\cong} \\ A\ar[drr]^{1\otimes {j'}} \ar[dddrr]_1^{~}="1"
&&&& A\ar[dll]_{j\otimes 1}
\ar[dddll]^1_{~}="2" 
\\ && A\otimes A\ar[dd]_p_{~}="3"
\ar@{=>} ;"2"_{\lambda}
\ar@{=>}"1";_{\rho'}
\\ 
\\ && A }
\end{equation*} 

is equal to

\hspace{.4in}
\begin{equation*}  
\xymatrix@!=0.9pc{
&& I \ar[lld]_j \ar[rrd]^{j'} \ar@{}[dd]|{\cong} \\ A\ar[drr]^{1\otimes {j'}} \ar[dddrr]_1^{~}="1"
&&&& A\ar@/_1pc/[dll]_{j\otimes 1}^{~}="4"
\ar@/^3pc/[dddll]^1_{~}="2"
\ar@/^1.6pc/[dll]^{j'\otimes 1}_{~}="5" 
\\ && A\otimes A\ar[dd]_p_{~}="3"
\ar@{=>} ;"2"-/va(-40)+1.8pc/_{\lambda'}
\ar@{=>}"1";_{\rho'}
\ar@{=>}"4";"5"_{\tau\otimes 1}
\\ 
\\ && A } 
\end{equation*} 
which, by properties of Gray monoids, is equal to 
\begin{equation*}  
\xymatrix{
&& I \ar@/_1.2pc/[lld]_j^{~}="4" \ar[rrd]^{j'}
\ar@/^1.2pc/[lld]^{j'}_{~}="5" \ar@{}[dd]|{\cong}
 \\ A\ar[drr]_{1\otimes {j'}} \ar[dddrr]_1^{~}="1"
&&&& A\ar[dll]_{j'\otimes 1}
\ar[dddll]^1_{~}="2" 
\\ && A\otimes A\ar[dd]_p_{~}="3"
\ar@{=>} ;"2"_{\lambda'}
\ar@{=>}"1";_{\rho'}
\ar@{=>}"4";"5"_{\tau}
\\ 
\\ && A } 
\end{equation*} 
which finally, by equation \eqref{muu}, is equal to 
\begin{equation*} 
\xymatrix {
I \ar[rr]^{j'} \ar@/_1.4pc/[dd]_j^{~}="1"
\ar@/^1.4pc/[dd]^{j'}_{~}="2"  && A \ar[dd]^1 \ar@{}[ddll]|{\equiv}
\\ 
\\ A\ar[rr]_1 && A
\ar@{=>}"1";"2"_{\tau}} 
\end{equation*}
\end{proof}  

A skew semimonoidale in a Gray monoid $\cal B$ is a {\em skew monoidale} in $\cal B$ if $\cal U$$(A,p,\alpha)$ is non-empty. We denote such a skew monoidale by $(A,p,j,\alpha,\lambda,\rho)$.

The results proven above now imply, as in the case of the previous section, the following:
\begin{prop}
The units of a skew monoidale $(A,p,j,\alpha,\lambda,\rho)$ are unique up to a unique isomorphism (if they exist).
\end{prop}

\begin{remark}
\textbf{(Theory of Skew Monoidales)}
It should not be a surprise that the elementary nature of the proofs in Section 2 carry over to this setting especially if we were to write the axioms for a skew monoidale not as pasting diagrams in a Gray monoid but as equations between the 1-cells $p$ and $j$. Consider equations \eqref{muu} and \eqref{mmid} rewritten as 
$$\vbox{\xymatrix{
p(1\otimes j)j \ar[rr]^-{\cong}  && p(j\otimes 1)j \ar[d]^-{\lambda j} \\
j \ar[rr]_-{1} \ar[u]^-{\rho j} && j}}
\qquad   \qquad 
\vbox{\xymatrix{
p(p\otimes 1)(1\otimes j \otimes 1) \ar[rr]^-{\alpha (1\otimes j \otimes 1)}  && p(1\otimes p)(1\otimes j \otimes 1) \ar[d]^-{p(1\otimes \lambda)} \\
p \ar[rr]_-{1} \ar[u]^-{p(\rho \otimes 1)} && p}}$$

These equations "look" like the corresponding equations \eqref{unitunit} and \eqref{midunit} from the previous section. In Houston's 2007 thesis \cite{H} there is defined a formal language for a collection of objects, 1-cells, 2-cells and equations between the 2-cells that admits an interpretation, or model, in a monoidal bicategory which he has called a "calculus of components". The calculus of components was used in \cite{H} to show that some results for pseudomonoids(= monoidales) follow formally using the formal language from the corresponding result in the cartesian monoidal 2-category \textbf{Cat}. As noted in \cite{H}, the formal language is not completely general, it applies to a collection of 1-cells of the form 
$\xymatrix {A_1\otimes \ar@{.}[r] & \otimes A_n \ar[r] & B}$ where these can then create, by tensoring and composition, composite 1-cells into a single target object such that the composite 1-cells are of the same form of the original collection. It was also noted in \cite{H} that the calculus of components should be regarded as a higher dimensional analogue of the typed languages for monoidal categories defined by C. Barry Jay in \cite{Jay}. 

Our only remark is that, using the calculus of components for a theory of skew monoidales as compared to a theory of
pseudomonoids(= monoidales) as in \cite{H}, the only real difference is that we need three basic 2-cells as opposed to six and five equations between the 2-cells as opposed to two (not counting the invertibilty equations). With this in mind we then recognise that the results and proofs of this section are formally identical to the proofs of the previous section. That is, the formal proof in the language or theory of skew monoidales is the "same" as the proof of the corresponding result for skew moniodal categories from the previous section. For example, the use of one of the derivation rules for the equations between the 2-cells called the naturality axiom in \cite{H} corresponds to our use of naturality in the proof of Propostion~\ref{unique}.
\end{remark}

\subsubsection*{\textbf{Acknowledgements:}}
I wish to thank my Masters supervisor, Stephen Lack, whose patience and guidance made possible the writing of this article.

\bibliographystyle{amsplain}

\end{document}